\documentclass[12pt,  amscd, amsfonts,leqno]{amsart}

      \theoremstyle{plain}
      \newtheorem{theorem}{Theorem}[section]
      \newtheorem{lemma}[theorem]{Lemma}
      \newtheorem{corollary}[theorem]{Corollary}
      \newtheorem{proposition}[theorem]{Proposition}

      \theoremstyle{definition}
      
  \newtheorem{example}[theorem]{Example}

      \theoremstyle{remark}
      \newtheorem{remark}[theorem]{Remark}

\begin{document}

\title{Extension of  Plurisubharmonic Functions in the Lelong Class }
\author{Ozcan Yazici}
\thanks{}
\subjclass[2010]{Primary 32U05; Secondary: 32C25, 32Q15, 32Q28}
\date{}
\address{}

\pagestyle{myheadings}
\begin{abstract} Let $X$ be an algebraic subvariety of $\mathbb C^n$ and $\overline X$ be its closure in $\mathbb P^n.$ In their paper \cite{CGZ}  Coman-Guedj-Zeriahi proved that  any plurisubharmonic function with logarithmic growth on $X$ extends to a plurisubharmonic function with logarithmic growth on $\mathbb C^n$ when  the germs $(\overline X,a)$ in $\mathbb P^n$ are irreducible for all $a\in \overline X\setminus X.$  In this paper we consider $X$  for which the germ $(\overline X,a)$  is reducible for some $a\in \overline X\setminus X$ and we give a necessary and sufficient condition for $X$ so that any plurisubharmonic function with logarithmic growth on $X$ extends to a plurisubharmonic function with logarithmic growth on $\mathbb C^n.$

\end{abstract}

\maketitle

\section{Introduction}

Let $X$ be an analytic subvariety of $\mathbb C^n$. A function $\phi : X \rightarrow [-\infty, +\infty)$ is called plurisubharmonic (psh) if $\phi\not\equiv -\infty $ on any open subset of $X$ and every point $z\in X$ has a neighborhood $U$ in $\mathbb C^n$ so that $\phi=u|_X$ for some psh function $u$ on $U$. We refer to \cite{FN} and \cite{D2} for more details and equivalent definitions.

We denote by $\mathbb P^n$ the complex projective space and consider the standard embedding $$z\in \mathbb C^n \hookrightarrow [1:z]\in \mathbb P^n,$$ where $[t:z]$ denote the homogeneous coordinates on $\mathbb P^n$. Let  $\omega$ be the Fubini-Study K\"{a}hler form on $\mathbb P^n$ with the potential function $\rho(t,z)=\log \sqrt{|t|^2+||z||^2}$. We call $\phi$ a $quasiplurisubharmonic$ ($qpsh$) function in $\mathbb P^n$ when $\phi$ is locally the sum of a psh function and a smooth function. Then the class of $\omega$- plurisubharmonic $(\omega-psh)$ functions on $\mathbb P^n$ is defined by $$PSH(\mathbb P^n,\omega)=\{\phi\in L^1(\mathbb P^n,[-\infty, +\infty)):\; \phi\; \text{ qpsh},\; dd^c\phi + \omega\geq 0   \},$$
where the operators $d$ and $d^c$ are defined by $d:=\partial +\overline \partial$, $d^c:=\frac{i}{2\pi}(\overline \partial -\partial)$ so that $dd^c=\frac{i}{\pi}\partial \overline \partial.$
We refer the reader to \cite{GZ} for the basic properties of $\omega$-psh functions.

Let  $X$ be an analytic subvariety of $\mathbb P^n$. An upper semicontinuous function $\phi: X\rightarrow [-\infty, +\infty)$ is called $\omega|_X-$ psh if $\phi\not\equiv -\infty$  on any open subset of $X$ and if there exist an open cover $\{U_i\}_{i\in I}$ of $X$ and psh functions $\phi_i$ and $\rho_i$ defined on $U_i$ where $\rho_i$ is smooth and $dd^c \rho_i=\omega$, so that $\rho_i+\phi=\phi_i$ holds on $X\cap U_i$ for all $i\in I.$ The class of $\omega|_X-$ psh functions on $X$ is denoted by $PSH(X,\omega|_X ).$

If not stated otherwise, we will assume in the sequel that $X$ is an algebraic subvariety of $\mathbb C^n$ for $n\geq 2$. By  $\overline{X}$  we denote the closure of $X$ in $\mathbb P^n$ so $\overline X$ is an algebraic subvariety of $\mathbb P^n$.
 By $\mathcal L (X)$ we denote the  $Lelong$ $class$ of psh functions $\phi$ on  $X$  which verify $\phi(z)\leq \log^+||z||+C \; \text{for all}\; z\in X,\; \text{where} \; C\; \text{is a constant that depends on}\; \phi .$

The mapping
\begin{eqnarray*} F: PSH(\mathbb P^n,\omega)\rightarrow \mathcal L(\mathbb C^n),\;
F(\phi)(z)=\rho(1,z)+\phi([1:z]),
\end{eqnarray*}
is well defined.
Its inverse $F^{-1}:\mathcal L(\mathbb C^n) \rightarrow  PSH(\mathbb P^n,\omega)$ is given by $F^{-1}(\eta)=\tilde \eta$, where
\begin{equation*}
\tilde \eta([t:z]) = \left\{
\begin{array}{ll}
\eta(z)-\rho(1,z) & \text{if} \;  t=1, \\
\limsup_{\mathbb C^n\ni[1:\zeta]\rightarrow [0:z]}(\eta(\zeta)-\rho(1,\zeta)) & \text{if } t=0.
\end{array} \right.
\end{equation*}
Thus there is a one-to-one correspondence between the Lelong class $\mathcal L(\mathbb C^n)$ and $PSH(\mathbb P^n,\omega)$  (see \cite{GZ} for details).
It is natural to ask that if there is such a correspondence between $\mathcal L(X)$ and $PSH(\overline X,\omega|_{\overline X})$. However  $F:PSH(\overline X,\omega|_{\overline X})\rightarrow \mathcal L(X)$ which is defined like above is not necessarily surjective.  In fact any $\eta\in \mathcal L(X)$ induces an upper semicontinuous function
\begin{equation*}
 \tilde \eta([t:z]): = \left\{
\begin{array}{ll}
\eta(z)-\rho(1,z) & \text{if} \; t=1, z\in X, \\
\limsup_{X\ni[1:\zeta]\rightarrow [0:z]}(\eta(\zeta)-\rho(1,\zeta)) & \text{if } t=0,
\end{array} \right.
\end{equation*}
 on $\overline X.$ However $\tilde\eta$ is not necessarily $\omega|_{\overline X}-psh$ on $\overline X.$ It is in general only $weakly \; \omega-psh$, i.e. it is bounded above on $\overline X$ and $\omega|_{\overline {X_r}}-psh$ on $\overline {X_r}$, where $\overline {X_r}$ is the regular part of $\overline X$ (see \cite{D2} for details).

 We denote by $\mathcal L_{\gamma}(X)$, where $\gamma$ is a positive constant, the Lelong class of psh functions on $X$ which verify $\phi(z)\leq \gamma\log^+||z||+C$ for all $z\in X$, where $C$ is a constant that depends on $\phi$. For an analytic subvariety $X\subset\mathbb C^n$  \cite[Theorem A]{CGZ} implies that any function $\phi\in \mathcal L(X)$  has an extension in $\mathcal L_{\gamma}(\mathbb C^n)$ for every $\gamma>1.$    In \cite[section 3]{CGZ} the question whether this additional arbitrarily small growth is  necessary  on an algebraic subvariety to have an extension is addressed. More precisely,  is every psh function with logarithmic growth on  an algebraic variety $X\subset \mathbb C^n$ the restriction of a function in $\mathcal L(\mathbb C^n)?$ The following is proved:

\begin{proposition}\label{1}\cite{CGZ} Let $\eta \in \mathcal L(X).$ The following are equivalent: \\
(i) There exists $\psi\in \mathcal L(\mathbb C^n)$ so that $\psi = \eta$ on $X.$ \\
(ii) $\tilde \eta \in PSH(\overline X, \omega|_{\overline X}).$\\
(iii) For every point $a\in \overline X \setminus X$ the following holds: if $(X_j,a)$ are irreducible components of the germ $(\overline X, a)$ then the value $$\limsup_{X_j\ni[1:\zeta]\rightarrow a}(\eta(\zeta)-\rho(1,\zeta)) $$  is independent of $j$. \\
In particular, if the germs $(\overline X,a)$ are irreducible for all points $a\in \overline X\setminus X$ then $\mathcal L (X)=\mathcal L (\mathbb C^n)|_X.$
\end{proposition}

Here  we consider the converse of the last statement : If $X$ is such that the germ $(\overline X,a)$ is reducible for some $a\in \overline X \setminus X$ then is there always a function in $\mathcal L (X)$ which does not extend to a function in $\mathcal L(\mathbb C^n)?$ That is, is the inclusion $\mathcal L (\mathbb C^n)|_X\subseteq \mathcal L(X)$ strict? 

We need to give a definition before we state our main result answering the above question.  Let $(X_k,a)$ be  irreducible components of the germ $(\overline X,a)$.  We will say that  $X_i$ and $X_j$  are \textit{linked} if there exist some irreducible components ${X_{i_k}}'s$ such that all the intersections $X_i\cap X_{i_1}\cap \mathbb C^n$, $X_{i_1}\cap X_{i_2}\cap \mathbb C^n$,..., $X_{i_m}\cap X_j\cap \mathbb C^n$ have positive dimension.
Now we can state our main result:

\begin{theorem} \label{main} Let $X$ be an algebraic variety in $\mathbb C^n$ where $n\geq 2$. Then any function in $\mathcal L(X)$ extends to a function in $\mathcal L(\mathbb C^n)$ if and only if for all $a\in \overline X\setminus X$, any two irreducible components of the germ $(\overline X,a)$ are linked.
\end{theorem}

The proof of this theorem will be given in Section 2. In section 3 we will consider some well known examples of algebraic varieties $X$ and  check whether our condition in Theorem \ref{main} holds for $X$, hence whether $\mathcal L(X)=\mathcal L(\mathbb C^n)|_X.$ In cases where $\mathcal L(X)\neq \mathcal L(\mathbb C^n)|_X$ we will construct a function $\eta\in\mathcal L (X)$ which has no extension in $\mathcal L(\mathbb C^n)$. By \cite[Theorem A]{CGZ} we know that there is an extension of $\eta$ in $\mathcal L_{\gamma}(\mathbb C^n)$ for all $\gamma>1.$  Here we give such an extension of $\eta$ explicitly.  If the germs $(\overline X,a)$ are irreducible for all points $a\in\overline X \setminus X$, then $\mathcal L (X)=\mathcal L(\mathbb C^n)|_X$ by Proposition \ref{1}. In particular, this happens when $\overline X$ is a smooth variety. In Example \ref{ex2} there are three singular points in $\overline X\setminus X.$ At two of these points the germs $(\overline X,a)$ are irreducible. At the other point  the germ $(\overline X,a)$ has two irreducible components which are linked. Thus by Theorem \ref{main},  $\mathcal L (X)=\mathcal L(\mathbb C^n)|_X$.  In Example \ref{ex3}, Example \ref{ex4} and Example \ref{ex5} for some singular point $a\in\overline X\setminus X$, the germ $(\overline X,a)$ has two irreducible components whose intersection lies in the hyperplane at infinity. Therefore these irreducible components are not linked and  by Theorem $\ref{main}$ $\mathcal L (X)\neq\mathcal L(\mathbb C^n)|_X$. In these cases we construct a function  $\eta\in\mathcal L(X)$ which has no extension in   $\mathcal L(\mathbb C^n)$. Then we give an explicit extension of $\eta$ in $\mathcal L_{\gamma}(\mathbb C^n)$ for any $\gamma>1.$\\

\noindent \textbf{Acknowledgments.} The author is grateful to Professor Dan Coman for his guidance and support. He would like to thank the referee for the careful reading of the paper and for comments improving the exposition.

\section {Proof of the Theorem \ref{main}}
 We need some lemmas to prove Theorem \ref{main}.

\begin{lemma} \label{sep}Let $X$ be as in Theorem \ref{main} and let $a\in \overline X\setminus X.$ If two irreducible components $X_i$ and $X_j$ of the germ $(\overline X,a)$ are not linked then  $ (\overline X,a)=\tilde X_i\cup \tilde X_j$  where  $\tilde X_i$ and $\tilde X_j$ are germs of subvarieties of $\overline X$ at $a$ such that $\tilde X_i\cap \tilde X_j\cap \mathbb C^n=\emptyset.$
\end{lemma}

\begin{proof} Let $(X_k,a)$, $k\in I$, be irreducible components of the germ $(\overline X,a)$. We take $\tilde X_i=X_i\cup \{\cup_{k\in K}X_k\} $ and $\tilde X_j=X_j\cup \{\cup_{k\notin K}X_k\}$ where $K=\{k\in I: X_k\;  \text{linked to}\;  X_i \}$. We claim that $\tilde X_i\cap\tilde X_j\cap \mathbb C^n $ has dimension $0$. Otherwise one of the irreducible component of $\tilde X_i$ is linked to an irreducible component of $\tilde X_j$. Consequently this irreducible component of $\tilde X_j$ is linked to $X_i$ and this contradicts the definition of the set $K$.\\
\indent Let $Y$ be an irreducible component of the germ $\tilde X_i\cap \tilde X_j$ at $a$. Then the previous claim implies that $Y\subset\{t=0\}$, hence $\tilde {X_i}\cap \tilde{X_j}\subset \{t=0\}.$ Thus  $\tilde X_i\cap \tilde X_j\cap \mathbb C^n=\emptyset.$\end{proof}

The following lemma will show that for a qpsh function $v$ on a germ of an irreducible analytic variety $(X,p)$, $\limsup_{z\rightarrow p} v(z)$  is attained along the complement $X\setminus Y$ for any proper germ of subvariety $(Y,p)$ of $(X,p).$   Although it is well known, we will include its proof for the convenience of the reader.

\begin{lemma}\label{lsp2}Let $(Y,p)\subset (X,p)$ be germs of analytic varieties in $\mathbb C^n$ such that  $\dim (X,p)=k>0$, $(X,p)$ is irreducible and $(Y,p)\neq (X,p).$ Then $$\limsup_{X\setminus Y\ni z\rightarrow p}v(z)=v(p),$$ for any $qpsh$ function $v$ on $(X,p).$
\end{lemma}

\begin{proof} We construct a non-constant holomorphic function $f:\Delta_{\epsilon}\to X$ such that $f(0)=p$ and $f(\Delta_{\epsilon})\cap Y=\{p\}$ where $\Delta_{\epsilon}$ is a disc of radius $\epsilon$ in $\mathbb C.$ Let $\pi:\mathbb C^k\times\mathbb C^{n-k}\rightarrow \mathbb C^k$ be the projection map onto the first $k$ coordinates. By Local Parametrization Theorem (see \cite[Theorem 4.19 on page 95]{D1}) there is a choice of coordinates in $\mathbb C^n$ such that the restriction of the projection map $\pi: X\cap U\rightarrow U'$ is a finite, proper, holomorphic map where $U$ and $U'$ are some  neighborhood of $p\in\mathbb C^n$ and $0\in \mathbb C^k$ with $\pi(p)=0.$  By Remmert's Proper Mapping Theorem (see \cite[Theorem 8.8 on page 118]{D1}) $\pi(Y\cap U)\subset U'$ is an analytic subvariety. Since $\pi$ is a finite map, $\dim\pi(Y\cap U)=\dim Y<k$ by \cite[Lemma 8.1 on page 118]{D1}.  Let $B_r\subset U'$ be a polydisc in $\mathbb C^k$ centered at $0$ with radius $r>0.$ Let $a\in B_{\frac{r}{2}}\setminus \pi(Y\cap U)$.  We define a holomorphic map $\phi$ from unit disc $\Delta\subset \mathbb C$ to $B_r\subset U'$  by $\phi(\zeta)=2a\zeta.$ Then $\phi(0)=0$ and $\phi(\frac{1}{2})=a\notin\pi(Y\cap U).$ Thus $\phi^{-1}(\pi(Y\cap U))$ is a proper subvariety of $\Delta\subset \mathbb C$. This implies that $0$ is an isolated point in $\phi^{-1}(\pi(Y\cap U))$. We take a smaller disc $\Delta_{\rho}$ such that $\phi(\Delta_{\rho})\cap \pi(Y\cap U)=\{0\}.$ $\pi^{-1}(\phi(\Delta_{\rho}))$ is an analytic subvariety of  $X\cap U$ and its dimension is 1 since $\pi$ is finite.  By parametrization of curves (see \cite[Example 4.27 on page 98]{D1}) there is a non-constant holomorphic map $$f:\Delta_{\epsilon}\subset\mathbb C\rightarrow \pi^{-1}(\phi(\Delta_{\rho}))\subset X,$$ with $f(0)=p$ for some disc $\Delta_{\epsilon}$ of radius $\epsilon.$ It follows that $f(\Delta_{\epsilon})\cap Y=\{p\}.$

Let $v$ be a $qpsh$ function on the germ $(X,p)$. Since $v$ is locally the sum of a psh function and smooth function, it is enough to prove the lemma when $v$ is psh. Since $v\circ f$ is subharmonic in $\Delta_{\epsilon}$ $$ v(p)=v(f(0))=\limsup_{0\neq t\rightarrow 0}v\circ f(t)\leq \limsup _{X\setminus Y\ni z\rightarrow p}v(z)\leq \limsup _{X\ni z\rightarrow p}v(z)=v(p).$$
\end{proof}

 \begin{proof}[Proof of Theorem \ref{main}] First we assume that the germ $(\overline X,a)$ has two irreducible components $X_i$ and $X_j$ which are not linked. Then we will show that   there exists a $\eta\in \mathcal L(X)$ which  has no extension in $\mathcal L (\mathbb C^n)$. For simplicity, we can assume that $a=[0:0:...:1]\in \{t=0\}\subset \mathbb P^n$. We take a neighborhood  $\overline V_M$  of $a$ where  $$V_M=\{(z_1,z_2,...,z_n)\in \mathbb C^n: |z_n|>M \max(1, |z_1|, |z_2|,...,|z_{n-1}| ) \}.$$ For $M$ big enough   $\ X \cap V_M=Y_1 \cup Y_2 $, and $\overline{Y_1}\cap \overline{Y_2}\subset \{t=0\}$ where we can choose  $\overline{Y_1}=\tilde X_i$ and $\overline{Y_2}=\tilde X_j$  as in  Lemma \ref{sep}. Let
\begin{eqnarray*}u(z_1,...,z_n):=\max\{\log|z_1|,\log |z_2|,...,\frac{1}{2}\log|z_n| \}.
\end{eqnarray*}
Note that $u\in \mathcal L(\mathbb C^n).$ We will show that \begin{eqnarray}\label{target}\label{rr}u(z)+2\log M\geq \rho(1,z) \; \text{on} \; \partial V_M, \end{eqnarray}
if $M$ is sufficiently  large.  In order to prove (\ref{target}) we consider four cases for $z\in \partial V_M.$

\indent $Case\;1:$ $|z_n|=M$ for $z=(z_1,...,z_n)\in \partial V_M.$ Then
 \begin{eqnarray*}\rho(1,z)&=&\frac{1}{2}\log(1+|z_1|^2+...+|z_n|^2) \\ &\leq&\frac{1}{2}\log(n+M^2)\leq\frac{1}{2}\log(2 M^2)\leq 2 \log M, \end{eqnarray*}
 when $M$ is big enough. On the other hand $u(z_1,...,z_n)=\frac{1}{2}\log M$ and inequality (\ref{target}) is satisfied.  
 
\indent $Case\; 2:$ $|z_n|=M|z_1|$ and $|z_n|<|z_1|^2.$ Then
$u(z_1,...,z_n)=\log|z_1|$  and
\begin{eqnarray*}\rho(1,z_1,...,z_n)&=&\frac{1}{2}\log(1+|z_1|^2+...+|z_n|^2)\\ &\leq& \frac{1}{2}\log((n+1)|z_n|^2)=\frac{1}{2}\log(n+1)+\log|z_n| \\ &\leq& 2 \log M+\log|z_1|.
\end{eqnarray*}
Thus inequality (\ref{target}) is satisfied in this case.

\indent $Case\;3:$ $|z_n|=M|z_1|$ and $|z_n|\geq |z_1|^2$. Then $$u(z_1,...,z_n)=\frac{1}{2}\log|z_n|=\frac{1}{2}(\log|z_1|+\log M)$$ and
\begin{eqnarray*}\rho(1,z_1,...,z_n)&\leq& \frac{1}{2}\log(n+1)+\log |z_n|=\frac{1}{2}\log(n+1)+\log |z_1|+\log M \\ &\leq& \frac{1}{2}\log|z_1|+\frac{3}{2}\log M +\frac{1}{2}\log(n+1),
\end{eqnarray*}
 since $|z_1|\leq M$. Therefore (\ref{target}) is satisfied in this case too. 
 
 \indent $Case\;4:$ $|z_n|=M|z_i|$ for some $i:2,...,n-1$. Then the same argument as above works. Thus we obtain the inequality (\ref{rr}). 

We consider \begin{equation*}
\eta(z_1,...,z_n) = \left\{
\begin{array}{ll}
\max(u(z_1,...,z_n)+ 2\log M, \rho(1,z_1,...,z_n)) & \text{on} \; Y_1,\\
u(z_1,...,z_n)+2 \log M & \text{on } X\setminus  Y_1.
\end{array} \right.
\end{equation*}

It follows using \cite[Theorem 5.3.1]{FN} that $\eta$ is psh on $X$.  Since $u\in \mathcal L(X)$,  $\eta \in \mathcal L(X).$ Let $$V_{M'}=\{(z_1:z_2:...:z_n)\in\mathbb C^n: |z_n|>M' \max(1, |z_1|, |z_2|,...,|z_{n-1}| ) \},$$  where $M'>e^{2k}M^4$ and $k$ is any positive number. In $V_{M'},$ $\rho(1,z)>\log|z_n|$ and
\begin{eqnarray*} u(z)+2\log M&\leq& \max(\log|z_n|-\log M',\frac{1}{2}\log |z_n|)+ 2\log M\\ &\leq &\log|z_n|-\frac{1}{2}\log M'+2\log M<\log |z_n|-k\leq \rho(1,z)-k.
\end{eqnarray*}
Thus $\eta(z)=\rho(1,z) $ on $Y_1$ near the point $a$ and $\eta(z) -\rho(1,z) <-k$ on $Y_2$ near  the point $a$. These imply that $$\limsup_{Y_1\ni[1:z_1:...:z_n]\rightarrow a}(\eta(z_1,...,z_n)-\rho (1,z_1,...,z_n))=0 $$ and $$ \limsup_{Y_2\ni[1:z_1:...:z_n]\rightarrow a}(\eta(z_1,...,z_n)-\rho (1,z_1,...,z_n))=-\infty.$$ Hence by Proposition \ref{1}, $\eta\in \mathcal L (X)$ does not extend in $\mathcal L (\mathbb C^n)$.  

Now we assume that any two irreducible components of the germ $(\overline X, a)$ are linked for any $a\in \overline X\setminus X$. Let $\eta\in \mathcal L(X)$ and  $X_i$ and $X_j$ be arbitrary irreducible components of $(\overline X,a)$. By the assumption there exist some irreducible components ${X_{i_k}}'s$ such that all the intersections $X_i\cap X_{i_1}\cap \mathbb C^n $, $X_{i_1}\cap X_{i_2}\cap \mathbb C^n $,..., $X_{i_m}\cap X_j\cap \mathbb C^n $ have positive dimension.
Let $C$ be a positive dimensional irreducible analytic subvariety of $X_i\cap X_{i_1}$ which is not contained in $\{t=0\}\subset \mathbb P^3.$ $\eta|_{X_i\cap\{t=1\}}$ induces a function $\tilde \eta$ on $X_i\subset\mathbb P^n$ defined by  
\begin{equation*}
\tilde\eta([t:z]) = \left\{\begin{array}{ll}
\eta(z)-\rho(1,z) & \text{if} \;t=1,\; z\in X_i\cap \mathbb C^n\\
\limsup_{X_i\ni[1:\zeta]\rightarrow [0:z]}(\eta(\zeta)-\rho(1,\zeta)) & \text{if}\; t=0,\; [0,z]\in X_i.
\end{array} \right.
\end{equation*}
Since $X_i$ is locally irreducible near $a$, \cite[Theorem 1.7]{D2} implies that $\tilde\eta$ is $\omega|_{X_i}$-psh on $X_i$. Then $\tilde\eta|_{ C}  $ is $\omega|_{C}$ - psh  on $C$. It follows that
\begin{eqnarray}\limsup_{X_i\ni[1:\zeta]\rightarrow a}(\eta(\zeta)-\rho(1,\zeta))&=&\tilde\eta(a)
=\tilde \eta|_C(a)=\limsup_{\mathbb C^n\cap C\ni[1:z]\rightarrow a}\tilde \eta|_C([1:z])\nonumber\\&=&\limsup _{C \ni [1:z]\rightarrow a}(\eta(z)-\rho(1,z))\nonumber. \end{eqnarray} Note that the third equality above follows from the Lemma \ref{lsp2}. By changing  $X_i$ with $X_{i_1}$  above, we obtain that $$\limsup_{X_{i_1}\ni[1:\zeta]\rightarrow a}(\eta(\zeta)-\rho(1,\zeta))=\limsup_{C\ni[1:\zeta]\rightarrow a}(\eta(\zeta)-\rho(1,\zeta))=\limsup_{X_i\ni[1:\zeta]\rightarrow a}(\eta(\zeta)-\rho(1,\zeta)). $$ By applying the same argument as above to the other irreducible components $X_{i_k}$ of the germ $(\overline X,a)$  we conclude that
$$\limsup_{X_i\ni[1:\zeta]\rightarrow a}(\eta(\zeta)-\rho(1,\zeta))=\limsup_{X_j\ni[1:\zeta]\rightarrow a}(\eta(\zeta)-\rho(1,\zeta)). $$
It follows from Proposition \ref{1} that $\eta\in \mathcal L(X)$ extends in $\mathcal L(\mathbb C^n).$
\end{proof}
\begin{remark} For $X\subset\mathbb C^2$ the intersection of two irreducible components of the germ $(\overline X,a) $ is  given by at most a finite set of points. Thus any two irreducible components of the germ the $(\overline X,a)$ are not linked when $(\overline X,a)$ is reducible. Therefore Theorem \ref{main} has the following immediate corollary in dimension two: \end{remark}
\begin{corollary} Let $X$ be an algebraic variety in $\mathbb C^2$. Then any function in $\mathcal L(X)$ extends to a function in $\mathcal L(\mathbb C^2)$ if and only if  the germs $(\overline X,a)$ are irreducible for all points $a\in \overline X\setminus X$.
\end{corollary} 

\section {Examples}
In this section we study some  well known examples  of algebraic varieties.
\begin{example} \label{ex2} Let  $X$ be the surface in $\mathbb C^3$ given by equation $xy^2+y^2z^2+z^2x^2-xyz=0$. Then $$\overline X=\{[t:x:y:z]\in \mathbb P^3: xty^2+y^2z^2+z^2x^2-xyzt=0 \}.$$ This surface is called Roman (Steiner) surface. Let $H=\{t=0\}\subset \mathbb P^3$ be the hyperplane at infinity. $H\cap\overline X=\{t=z=0\}\cup \{t=x^2+y^2=0 \}.$ One checks that  $[0:1:0:0]$, $[0:0:1:0]$ and $[0:0:0:1]$ are the only singular points of $\overline X$  in the hyperplane at infinity.

Near $a=[0:1:0:0]$, $\overline X$ is given by the zero set of the Weierstrass polynomial $f\in\mathcal O_{\mathbb C^2,0}[z]$, 
\begin{eqnarray*}\{(t,y,z)\in \mathbb C^3: f(t,y,z)=z^2-\frac{yt}{1+y^2}z+\frac{y^2t}{1+y^2}=0\}.\end{eqnarray*}
We show that the germ $(\overline X,a)$ is irreducible by proving that $f$ is irreducible as a Weierstrass polynomial. Indeed, if $f$ is reducible then $$f(t,y,z)=(z-g(t,y))(z-h(t,y)),$$ where $g\in\mathcal O_{\mathbb C^2,0}$ and $ h\in \mathcal O_{\mathbb C^2,0}$ with $g(0,0)=h(0,0)=0.$  It follows that
\begin{equation}\label{m1} gh=\frac{y^2t}{1+y^2} \;\text{and}\; g+h=\frac{yt}{1+y^2}.
\end{equation}
 Since $t$ divides $gh$, we may assume without loss of generality that $t|g$. By the second equation in (\ref{m1}) $t|h$. This implies that $t^2|gh$ which contradicts to the first equation  in (\ref{m1}). Thus the germ $(\overline X,a)$ is irreducible. In the same way one shows that the germ $(\overline X,a)$ is irreducible at $a=[0:0:1:0].$

Let $a=[0:0:0:1]$. Near $a$,  $\overline X$ is given by 
\begin{eqnarray*}\{(t,x,y)\in\mathbb C^3: f(t,x,y)=y^2-\frac{xt}{1+xt}y+\frac{x^2}{1+xt}=0\},
\end{eqnarray*}
near the origin in $\mathbb C^3.$ The function $f$ can be written as $f=f_1f_2$ where 
\begin{eqnarray*} f_1(t,x,y)=y-\frac{x}{2(1+xt)}(t +\sqrt{t^2-4xt-4} ),
\end{eqnarray*}
\begin{eqnarray*}f_2(t,x,y)=y-\frac{x}{2(1+xt)}(t -\sqrt{t^2-4xt-4} ).\end{eqnarray*} 
Here we take a branch of root function  with $\sqrt{-4}=2i.$ One checks that $\{f_1=0\}\cap \{f_2=0\}=\{x=y=0\}$ near $(0,0,0)$. Thus the germ $(\overline X,a)$ has two irreducible components whose intersection lies along a line not contained in the hyperplane at infinity. Therefore these irreducible components are linked and  by Theorem \ref{main}, $\mathcal L(X)=\mathcal L(\mathbb C^3)|_X.$
\end{example}

The following example is a generalization of \cite[Example 3.2]{CGZ}.

\begin{example} \label{ex3} Let $z=(z_1,...,z_n)\in \mathbb C^n$ and $X=\{z_{m+1}=...=z_n=0\}\cup \{z_{m+1}=...=z_n=1\}\subset \mathbb C^n$ be an $m$ dimensional  subvariety of $\mathbb C^n$.   Let $\rho(t,z)=\log \sqrt{|t|^2+||z||^2}$. The function  \begin{equation*}
\eta(z) = \left\{\begin{array}{ll}
\rho(1,z) & \text{if}\; z\in X_1=\{z_{m+1}=...= z_{n}=0\}, \\
\rho(1,z) + 1 & \text{if}\;  z\in X_2=\{ z_{m+1}=...= z_{n}=1\},
\end{array} \right. \end{equation*}
is in $\mathcal L(X)$ and $$ \limsup_{X_1\ni[1:\zeta]\rightarrow a}(\eta(\zeta)-\rho(1,\zeta))=0, \; \limsup_{X_2\ni[1:\zeta]\rightarrow a}(\eta(\zeta)-\rho(1,\zeta))=1, $$ where $a=[0:1:0:..:0]\in \overline X\setminus X$.  Proposition \ref{1} implies that $ \eta$ does not extend in $\mathcal L(\mathbb C^n)$.  However by \cite[Theorem A]{CGZ}, we can find an extension with arbitrarily small additional growth. Explicitly we take $$\tilde \eta(z)=\rho(1,z)+\epsilon\log|1+z_n(e^{\frac{1}{\epsilon}}-1)|.$$ Then $\tilde \eta\in \mathcal L_{1+\epsilon}(\mathbb C^n)$ and $\tilde \eta|_X=\eta.$
\end{example}

The following example is a generalization of \cite[Example $3.3$]{CGZ}.

\begin{example}\label{ex4} Let $X$ be given by the equation $z_1z_n=z_1^3+1.$ It is clear that $X$ is irreducible in $\mathbb C^n$. The closure $\overline X$ of $X$ in $\mathbb P^n$ is given by $$\overline{X}=\{[t:z_1:...:z_n]\in \mathbb P^n: z_1z_nt=z_1^3+t^3 \}=X\cup \{t=z_1=0\}.$$ 

We take $a=[0:...:0:1]\in \overline X \setminus X$. First we will show that the germ $(\overline X,a)$ has two irreducible components $X_1$ and $X_2$ whose intersection lies in the hyperplane at infinity. Let $(s_0,...,s_{n-1})$ be affine coordinates near $a\in \{z_n\neq 0\}$ where $s_0=\frac{t}{z_n}$, $s_i=\frac{z_i}{z_n}.$  In these coordinates the germ $(\overline X,a)$ is defined by $s_0s_1=s_0^3+s_1^3.$ We change the coordinate $s_0$ by $u=s_0+s_1$. In the new coordinates the germ $(\overline X,a)$ is defined by the Weierstrass polynomial $f(u,s_1,...,s_{n-1})=s_1^2-s_1u+\frac{u^3}{3u+1}=0$ and $f=f_1f_2$ where $f_1$ and $f_2$ are germs of holomorphic functions in $\mathcal O_{\mathbb C^n,0}$ defined by
\begin{eqnarray*} f_1(u,s_1,...,s_{n-1})=s_1-\frac{1}{2}u\left( 1+\sqrt{1-\frac{4u}{3u+1}}  \right),\end{eqnarray*}
\begin{eqnarray*}  f_2(u,s_1,...,s_{n-1})=s_1-\frac{1}{2}u \left( 1-\sqrt{1-\frac{4u}{3u+1}} \right).
\end{eqnarray*}
Then $\{f_1=0\}\cap\{f_2=0\}=\{s_1=u=0\}$ near $a$. In the original coordinates, $\{f_1=0\}\cap \{f_2=0\}=\{s_0=s_1=0\}$. Thus the germ $(\overline X,a)$ has two irreducible components $X_1$ and $X_2$ whose intersection is contained in the hyperplane at infinity. Therefore $X_1$ and $X_2$ are not linked and by Theorem \ref{main}, $\mathcal L(X)\neq \mathcal L(\mathbb C^n)|_X.$ 

Now we will give an explicit example of a function in $\mathcal L(X)$ which has no extension in $\mathcal L(\mathbb C^n).$ The function $$\psi(z)=\max(\log|z_n-z_1^2|,2\log|z_1|+1)$$ is psh in $\mathbb C^n$. It is clear that $\eta:=\psi|_{X}\in \mathcal L(X)$ since $z_1z_n=z_1^3+1.$ On $X$, in the coordinates $(s_0,...,s_{n-1})$ near $a$, we have $s_1\neq 0$ as $z_1\neq 0$ and $s_0\neq 0$ since $X\subset \mathbb C^n=\{t\neq 0\}.$  In these coordinates near $a$, the functions $\eta$ and $\rho$ are given by
\begin{eqnarray*} \eta(s_0,...,s_{n-1})=\max \left(\log\left |\frac{s_0}{s_1}\right |, 2\log\left |\frac{s_1}{s_0}\right |+1\right ),
\end{eqnarray*}
and
\begin{eqnarray*} \rho(s_0,...,s_{n-1})= \log\left(1+\left | \frac{s_1}{s_0} \right |^2+...+\left | \frac{s_{n-1}}{s_0} \right |^2+\frac{1}{|s_0|^2} \right)^{\frac{1}{2}} =\log \frac{1}{|s_0|} +o(1),
\end{eqnarray*}
as $(s_0,...,s_{n-1})\rightarrow (0,...,0).$

On $(X_1,a)$ in the coordinates $(s_0,...,s_{n-1}),$
\begin{eqnarray*} f_1(s_0,...,s_{n-1})=-s_0-\frac{1}{2}(s_0+s_1)O(|s_0+s_1|)=0.
\end{eqnarray*}
This implies that
\begin{eqnarray*} \left |\frac{s_1}{s_0}\right |=\frac{2+O(|s_0+s_1|)}{O(|s_0+s_1|)}\rightarrow \infty,
\end{eqnarray*}
as $(s_0,s_1)\rightarrow (0,0).$
Since $s_1^2/s_0=1-s_0^2/s_1$ on $(\overline X,a)$,
\begin{eqnarray*} \limsup_{X_1 \ni[1:\zeta] \rightarrow a}(\eta(\zeta)-\rho(1,\zeta))&=& \limsup_{(s_0,s_1)\rightarrow (0,0)}\left(2\log\left| \frac{s_1}{s_0} \right | +1-\log \frac{1}{|s_0|}\right)\\ &=&\limsup _{(s_0,s_1)\rightarrow (0,0)}  \log \left |\frac{s_1^2}{s_0}  \right |+1=1.
\end{eqnarray*}
Similarly we obtain that 
\begin{eqnarray*}\limsup_{ X_2 \ni[1:\zeta]\rightarrow a}(\eta(\zeta)-\rho(1,\zeta))=\limsup_{(s_0,s_1)\rightarrow (0,0)}\left(\log\left| \frac{s_0}{s_1} \right |-\log \frac{1}{|s_0|}\right)=0.
\end{eqnarray*}

Thus by Proposition \ref{1}, $\eta$ has no extension in $\mathcal L (\mathbb C^n).$
However we know from \cite[Theorem A]{CGZ} that with an arbitrarily small additional growth $\eta$ has an extension in $\mathcal L_{\gamma} (\mathbb C^n)$ where $\gamma>1.$ We will give an explicit extension using a similar idea given in the proof of \cite[Proposition 3.3]{BL}.

Let $\Theta=\{z\in\mathbb C^n: |z_1z_n-z_1^3-1|<e^{-3}\}$. Clearly $X\subset \Theta$. In $\bar {\Theta}\subset \mathbb C^n$ when $|z_1|<2$, $\psi$ has logarithmic growth.  In $ \bar {\Theta}$ when  $|z_1|\geq2$, $\psi(z)=\log|z_1|^2+1$ and  
\begin{eqnarray}\label{1245} |z_1|^2-\frac{\delta+1}{2}<|z_n|< |z_1|^2+\frac{\delta+1}{2},
\end{eqnarray}
where $\delta=e^{-3}.$ Thus $\psi$ has logarithmic growth in $\bar \Theta.$ One can easily check that $\psi(z)\leq\log^+||z||+3$ in $\bar {\Theta}.$   Indeed, if $z\in \Theta$ and $|z_1|<2$ then 
\begin{eqnarray*} \psi(z)\leq\max(\log(|z_n|+4),\log 4+1)\leq \max(\log(5|z_n|),3)\leq \log^+||z||+3.
\end{eqnarray*}
If $z\in \Theta$ and $|z_1|\geq 2$ then inequality (\ref{1245}) implies that 
\begin{eqnarray*} \psi(z)&=&\log|z_1|^2+1\leq\log(|z_n|+1)+1\\&\leq& \max(\log 2,\log(2|z_n|))+1\leq \log^+||z||+3.
\end{eqnarray*}
Let $\phi(z)=\epsilon(\frac{1}{3}\log|z_1z_n-z_1^3-1|+1)$ where $\epsilon>0.$ Then $\phi\in \mathcal L_{\epsilon}(\mathbb C^n)$, $\phi=-\infty$ on $X$ and $\phi\geq 0$ on $\mathbb C^n \setminus \Theta. $ We now define
\begin{equation*}
\tilde \eta(z) = \left\{
\begin{array}{ll}
\max(\psi(z), \log^+||z||+\phi(z)+3) & \text{if} \;z\in \Theta,\\
\log^+||z||+\phi(z)+3 & \text{if}\; z\in\mathbb C^n\setminus \Theta.
\end{array} \right.
\end{equation*}
We have $\psi(z)\leq \log^+||z||+\phi(z)+3$  on $\partial \Theta$. So $\tilde \eta\in PSH(\mathbb C^n).$ Since $\psi\in \mathcal L(\Theta)$ and $\phi \in \mathcal L_{\epsilon}(\mathbb C^n)$, $\tilde \eta\in \mathcal L_{1+\epsilon}(\mathbb C^n)$ and $\tilde \eta|_X=\eta.$
\end{example}

\begin{example}\label{ex5} Let  $X$ be the surface in $\mathbb C^3$ given by equation $3z-3xy+x^3=0.$ Then $$\overline X=\{[t:x:y:z]\in \mathbb P^3: 3zt^2-3xyt+x^3=0 \}.$$ This surface is called Cayley's ruled cubic surface. One verifies that $\overline X\cap H=\{t=x=0 \}$ and all the points on the line $\{t=x=0\}$ are singular.    

Let  $a=[0:0:1:1]\in \overline X\setminus X.$ In the neighborhood $\{z\neq 0\}\subset \mathbb P^3$ of $a$, $\overline X$ is given by $$\{(t,x,y)\in \mathbb C^3: f(t,x,y)=t^2-xyt+\frac{x^3}{3}=0\},$$ near $(0,0,1)$. The function $f$ can be written as $f(t,x,y)=f_1(t,x,y)f_2(t,x,y)$ where $$f_1(t,x,y)=t-\frac{xy}{2}\left(1-\sqrt{1-\frac{4x}{3y^2}}\right) \;\text{and}\; f_2(t,x,y)=t-\frac{xy}{2}\left(1+\sqrt{1-\frac{4x}{3y^2}}\right)$$ are holomorphic near $(0,0,1).$  In a small neighborhood of $(0,0,1),$ $$\{f_1=0\}\cap\{f_2=0\}=\{t=x=0\}.$$ Thus the germ $(\overline X,a)$ has two irreducible components whose intersection lies along a line contained in the hyperplane at infinity. Therefore these irreducible components are not linked and by Theorem \ref{main}, $\mathcal L(X)\neq \mathcal L(\mathbb C^3)|_X.$ 

Now we will give an explicit example of a function in $\mathcal L(X)$ which does not extend to $\mathbb C^3$ with logarithmic growth. Let $a=[0:0:1:1]$ and $X_1$, $X_2$ be the two irreducible components of  the germ $(\overline X, a)$ defined by $f_1$ and $f_2$ respectively.

Let $C_1$ be the curve $\{(x,\frac{x^2+x}{3},\frac{x^2}{3}):x\in \mathbb C\}\subset X$ and $\overline C_1$ be its closure in $\mathbb P^3$. One checks that  $a\in \overline C_1$ and $(\overline C_1,a)\subset X_1.$

Let $C_2$ be the curve $\{(1,y,y-\frac{1}{3}):y\in\mathbb C\}\subset X$. One checks that $a\in \overline C_2$ and $(\overline C_2,a)\subset X_2.$ 

We now define $\psi(x,y,z)=\max(\log|z|, 2\log|x|)\in PSH(\mathbb C^3).$ First we show that $\eta:=\psi|_X\in \mathcal L(X).$ When $|z|>|x|^2$, $\eta(x,y,z)=\log|z|.$ Hence we may assume that $|z|\leq|x|^2.$ It follows that  $\eta(x,y,z)=\log|x|^2$ and $|y|>\frac{|x|^2}{3}-\frac{|z|}{|x|}>\frac{|x|^2}{3}-|x|>\frac{|x|^2}{6}$ when $|x|>6$ on $X$.
Thus $\eta\in \mathcal L(X).$ 

Note that $\overline {C_1}=C_1\cup \{a\}$ and $X_1\supset C_1\ni(x,\frac{x^2+x}{3},\frac{x^2}{3}) \rightarrow a$ as $x\rightarrow \infty.$ $\eta|_{X_i\cap\{t=1\}}$ induces a function $\tilde \eta_i$ on $X_i\subset\mathbb P^3$ defined near $a$ by  
\begin{equation*}
\tilde{\eta}_i([t:\tau]) = \left\{\begin{array}{ll}
\eta(\tau)-\rho(1,\tau) & t=1, \\
\limsup_{X_i\cap\mathbb C^3\ni[1:\zeta]\rightarrow [0:\tau]}(\eta(\zeta)-\rho(1,\zeta)) & t=0,
\end{array} \right. \end{equation*}
where $\tau=(x,y,z).$ Since $X_i$ is locally irreducible near $a$, \cite[Theorem 1.7]{D2} implies that $\tilde\eta_i$ is $\omega|_{X_i}-psh$ on $X_i$ and $\tilde \eta_i|_{\bar{C_i}}$ is $\omega|_{\bar{C_i}}-psh$ on $\bar{C_i}.$ Then
\begin{eqnarray*}
&&\limsup_{X_1\cap \mathbb C^3\ni[1:\zeta]\rightarrow a}(\eta(\zeta)-\rho(1,\zeta))=\tilde{\eta_1}(a)=\tilde{\eta_1}|_{\bar{C_1}}(a)= \limsup_{C_1\ni\tau\rightarrow a}\tilde{\eta_1}|_{\bar{C_1}}(\tau)\\
&& \ \ \ \ \ =\limsup_{x\rightarrow \infty}\left(\eta \left(x,\frac{x^2+x}{3},\frac{x^2}{3}  \right )-\rho\left(1,x,\frac{x^2+x}{3},\frac{x^2}{3} \right)\right)=\log\left(\frac{3}{\sqrt 2}\right). \end{eqnarray*}
The third equality holds by Lemma \ref{lsp2}. Similarly
\begin{eqnarray*}
&&\limsup_{X_2\cap \mathbb C^3\ni[1:\zeta]\rightarrow a}(\eta(\zeta)-\rho(1,\zeta))=\tilde{\eta_2}(a)=\tilde{\eta_2}|_{\bar{C_2}}(a)= \limsup_{C_2\ni\tau\rightarrow a}\tilde{\eta_2}|_{\bar{C_2}}(\tau)\\&&\ \ \ \ \ =\limsup_{y\rightarrow \infty}\left(\eta \left(1,y,y-\frac{1}{3}  \right )-\rho\left(1,1,y,y-\frac{1}{3}\right)\right)=\log\left(\frac{1}{\sqrt 2}\right). \end{eqnarray*}

By Proposition \ref{1}, $\eta$ does not extend in $\mathcal L(\mathbb C^3).$ 

We know from \cite[Theorem A]{CGZ} that with an arbitrarily small additional growth $\eta$ extends in $\mathcal L_{\gamma}(\mathbb C^3)$ where $\gamma>1$. We will give an explicit extension using a similar idea given in the proof of \cite[Proposition 3.3]{BL}.

Let $\Theta=\{(x,y,z)\in \mathbb C^3: |3z-3xy+x^3|<e^{-3} \}.$ Clearly  $X$ is contained in $\Theta$. First we show that $\psi\in \mathcal L(\Theta).$ When $|x|<6$, $\psi(x,y,z)\leq \log|z|+4.$ When $|z|>|x|^2$, $\psi(x,y,z)=\log|z|.$ Hence we can assume that $|z|\leq|x|^2$ and $|x|\geq 6.$ It follows that $|y-\frac{z}{x}-\frac{x^2}{3}|<\frac{e^{-3}}{18}.$ Therefore $|y|>\frac{|x|^2}{12}$ and $$\psi(x,y,z)=\log|x|^2\leq \log|y|+\log12. $$ Thus in $\Theta$, $\psi(x,y,z)\leq \log||(x,y,z)||+4.$ That is, $\psi\in \mathcal L(\Theta).$ Let $$\phi(x,y,z)=\epsilon\left(\frac{1}{3}\log|3z-3xy+x^3|+1 \right),$$ where $\epsilon>0.$ Then $\phi \in \mathcal L_{\epsilon}(\mathbb C^3)$, $\phi=-\infty$ on $X$ and $\phi\geq 0$ on $\mathbb C^3\setminus \Theta.$ We now define
\begin{equation*}
\tilde \eta(\tau) = \left\{
\begin{array}{ll}
\max(\psi(\tau), \log||\tau||+\phi(\tau)+4) & \text{if}\; \tau=(x,y,z)\in \Theta,\\
\log||\tau||+\phi(\tau)+4 & \text{if}\; \tau\in \mathbb C^3\setminus \Theta.
\end{array} \right.
\end{equation*}
Since we have $\psi(\tau)\leq \log||\tau||+\phi(\tau)+4$ on $\partial \Theta$, $\tilde \eta\in PSH(\mathbb C^3).$ As $\psi\in \mathcal L (\Theta)$ and $\phi\in \mathcal L_{\epsilon}(\mathbb C^3),$ $\tilde \eta \in \mathcal L_{1+\epsilon}(\mathbb C^3)$ and $\tilde \eta|_X=\eta.$
\end{example}

 \end{document}